\documentclass[11pt]{amsart}

\usepackage{amsmath}
\usepackage{amsthm}
\usepackage{amssymb}
\usepackage{amscd}          			
\usepackage{bbm}            			
\usepackage{mathrsfs}           		
\usepackage{mathalfa}

\usepackage{graphicx}          		

\usepackage{verbatim}           		
\usepackage{enumitem}      		

\theoremstyle{plain}
\newtheorem{Theorem}{Theorem}

\newtheorem{theorem}[Theorem]{Theorem}

\newtheorem{corollary}[Theorem]{Corollary}

\newtheorem{lemma}[Theorem]{Lemma}

\theoremstyle{definition}
\newtheorem{Definition}[Theorem]{Definition}

\newtheorem{example}[Theorem]{Example}
\newtheorem{definition}[Theorem]{Definition}
\newtheorem{remark}[Theorem]{Remark}

\theoremstyle{remark}

\newcommand{\N}{\mathbb{N}}     
\newcommand{\R}{\mathbb{R}}     
 
\def\r{\R}

\newcommand{\calC}{\mathscr{C}}

\newcommand{\calK}{\mathscr{K}}

\newcommand{\calO}{\mathscr{O}}

\DeclareMathOperator{\cl}{cl}				
\DeclareMathOperator{\supp}{supp}			
\DeclareMathOperator{\rg}{rg}				

\newcommand{\ox}{\calO(X)}
\newcommand{\cx}{\calC(X)}
\newcommand{\kx}{\calK(X)}

\newcommand{\bcx}{\kx}

\newcommand{\fvix}{C_0(X)}  				
\newcommand{\fvixP}{C_0^+(X)}                          
\newcommand{\fcsx}{C_c(X)}    			
\newcommand{\fcsxP}{C_c^+(X)}                       
 
\def\K{{{\mathcal{K}}}}

\renewcommand{\O}{\emptyset}

\def\sm{\setminus}
\def\cl{\overline}

\def\se{\subseteq}
\def\sc{\sqcup}
\def\bsc{\bigsqcup}

\def\rf{\mathcal{R}}

\def\rg{\rho_{g}}

\def\eps{\epsilon}

\def\la1{\lambda_1}
\def\la2{\lambda_2}
\def\la0{\lambda_{0}}
\def\la{\lambda}

\def\mugt{{\widehat{\mu_g}}}

\usepackage{times}

\begin{document}
\title{Integration with respect to deficient topological measures on locally compact spaces}
\author{S. V. Butler, UCSB } 
\address{Department of Mathematics,
University of California Santa Barbara, 
552 University Rd., Isla Vista, CA 93117, USA } 
\email{svbutler@ucsb.edu }  
\date{February 21, 2019}
\maketitle

\begin{abstract}
Topological measures and deficient topological measures generalize Borel measures and correspond to certain non-linear functionals. 
We study integration with respect to deficient topological measures on locally compact spaces. 
Such an integration over sets yields a new deficient topological measure 
if we integrate a nonnegative vanishing at infinity function; and it produces a signed deficient 
topological measure if we use a continuous function on a compact space.  
We present many properties of these resulting deficient topological measures and of signed 
deficient topological measures. In particular, they are  absolutely continuous with respect to 
the original deficient topological measure and Lipschitz continuous.  Deficient topological measures obtained by integration over sets 
can also be obtained from non-linear functionals.
We show that for a deficient topological measure $ \mu$ that assumes finitely many values, there is a function $ f $ such that 
$\int_X f \, d \mu = 0$, but $\int_X (-f )\, d \mu  \neq 0$. We present different criteria for $\int_X f \, d \mu = 0$.
We also prove some convergence results, including a Monotone convergence theorem. 
\end{abstract}

\section{Introduction}

Topological measures (initially called quasi-measures) were introduced by 
J. F. Aarnes in \cite{Aarnes:TheFirstPaper}, \cite{Aarnes:Pure}, and \cite{Aarnes:ConstructionPaper}.
These generalizations of measures are defined on open and closed subsets of a topological space. Despite the lack of an algebraic structure on the domain, 
absence of subadditivity, and unavailability of many standard techniques of measure theory and functional analysis, topological measures
still possess many features of regular Borel measures. 
Via integration, topological measures correspond to functionals that are not linear, but are linear on singly 
generated subalgebras. 
Topological measures and corresponding quasi-linear functionals are connected to the problem 
of linearity of the expectational functional on the algebra of observables in quantum mechanics. Initial papers by Aarnes were followed 
by many papers devoted to the subject.
Applications of topological measures and corresponding non-linear functionals to symplectic topology 
have been studied in numerous papers beginning with \cite{EntovPolterovich} (which has been cited over 100 times), 
and in a monograph \cite{PoltRosenBook}. 

Topological measures are a subclass of deficient topological measures, which also correspond via integration to certain non-linear functionals.
(See, for example,  \cite{OrjanAlf:CostrPropQlf}, \cite{Svistula:Signed}, \cite{Svistula:DTM}, \cite{Butler:DTMLC},  and \cite{Butler:ReprDTM} for more information).
Deficient topological measures are not only interesting by themselves, but also provide an essential framework for studying topological measures and 
various non-linear functionals. This provided a motivation for our study of integration 
with respect to deficient topological measures on locally compact spaces, 
especially given the fact that the vast majority of results so far has been proven for compact spaces 
(which  has impeded the development of the area and its applications).  
We demonstrate that integration with respect to a deficient topological measure sometimes gives the same results as integration with respect to a measure 
would give; in other instances, the results are very different. Some of our results are new, and some are generalizations of known results about 
deficient topological measure on compact spaces to signed deficient topological measures on compact spaces and /or deficient topological measures on 
locally compact spaces. (See \cite{Butler:STMLC} and \cite{Butler:Decomp}  for more information about signed deficient topological measures.)

We begin (Section \ref{IntlSetRestr}) with the study of integration with respect to a deficient topological measure over a set. 
It is done by using "restricted" deficient topological measures.
Integration over sets yields a new deficient topological measure if we integrate a nonnegative vanishing at infinity function; and it produces a signed deficient 
topological measure if we use a continuous function on a compact space.  
We present many properties of these resulting deficient topological measures and of signed 
deficient topological measures. In particular, these resulting (signed) deficient topological measures are absolutely continuous with respect to 
the original deficient topological measure, and they are Lipschitz continuous.
In Section \ref{IntSetFnl} we show that the deficient topological measures obtained in Section \ref{IntlSetRestr} from finite deficient topological measures 
by integrating over sets can also be obtained from non-linear functionals. 
We present more properties of such deficient topological measures.   
In Section \ref{ZeroInt} we show that for a deficient topological measure $ \mu$ that assumes finitely many values, there is $ f \ge 0$ such that 
$\int_X f \, d \mu = 0$, but $\int_X (-f )\, d \mu  \neq 0$. Using integration over zero and cozero sets, we present different criteria for $\int_X f \, d \mu = 0$.
We conclude the paper (Section \ref{MCT}) with some convergence results, including a Monotone convergence theorem. 
 
In this paper $X$ is a locally compact, connected space. 

By $C(X)$ we denote the set of all real-valued continuous functions on $X$ with the uniform norm, 
by $C_b(X)$ the set of bounded continuous functions on $X$,   
by $C_0(X)$ the set of continuous functions on $X$ vanishing at infinity,  and
by $C_c(X)$ the set of continuous functions with compact support. 
By $C_0^+(X)$ and  $C_c^+(X)$  we denote the collection of all nonnegative functions from $C_0(X)$ and $C_c(X)$, respectively.
$Z(f)$ and $Coz(f)$ stand for the zero and the cozero sets of the function $f$, i.e. 
$ Z(f) = \{ x \in X: \, f(x) = 0 \}, \ \ Coz(f) = X \sm Z(f).$

When we consider maps  into $ [-\infty, \infty]$ we assume that any such map 
attains at most one of $ \infty, - \infty$, and is not identically $\infty$ or $ - \infty$. 

We denote by $\overline E$ the closure of a set $E$, and by $ \bigsqcup$ a union of disjoint sets.
We denote by $1$ the constant function $1(x) =1$,  by $id$ the identity function $id(x) = x$, 
and by $1_K$ the characteristic function of a set $K$. By $ supp \,  f $ we mean $ \overline{ \{x: f(x) \neq 0 \} }$.

Several collections of sets are used often.   They include:
$\mathscr{O}(X)$;
$\mathscr{C}(X)$; and
$\mathscr{K}(X)$-- 
 the collection of open subsets of $X$;  the collection of closed subsets of   $X $;
and the collection of compact subsets of   $X $, respectively.


\begin{definition} \label{MDe2}
Let $X$ be a  topological space and $\nu$ be a set function on a family of subsets of $X$ that contains $\mathscr{O}(X) \cup \mathscr{C}(X)$. 
We say that 
\begin{itemize}
\item
$\nu$ is compact-finite if $ |\nu(K) | < \infty$ for any $ K \in \mathscr{K}(X)$;
\item
$\nu$ is simple if it only assumes  values $0$ and $1$;
\item
a nonnegative set-function $ \nu$ is finite if $ \nu(X) < \infty$.
\end{itemize}
\end{definition}

\begin{definition}
A Radon measure  $m$  on $X$ is  a Borel measure that is finite on all compact sets,
outer regular on all Borel sets, and inner regular on all open sets, i.e.
$ m(E) = \inf \{ m(U): \ \ E \subseteq U, \ U \text{  open  } \} $ for every Borel set $E$, and 
$m(U) = \sup \{  m(K): \ \ K \subseteq U, \  K  \text{  compact  } \}$ for every open set $U$. 
\end{definition}

Recall the following fact (see, for example,~\cite[Chapter XI, 6.2]{Dugundji}):
\begin{lemma} \label{easyLeLC}
Let $K \subseteq U, \ K \in \mathscr{K}(X),  \ U \in \mathscr{O}(X)$ in a locally compact space $X$.
Then there exists a set  $V \in \mathscr{O}(X)$ with compact closure such that
$ K \subseteq V \subseteq \overline V \subseteq U. $
\end{lemma}

\begin{Definition}\label{DTM}
A  deficient topological measure on a locally compact space $X$ is a set function
$\nu:  \cx \cup \ox \longrightarrow [0, \infty]$ 
which is finitely additive on compact sets, inner compact regular, and 
outer regular, i.e. :
\begin{enumerate}[label=(DTM\arabic*),ref=(DTM\arabic*)]
\item \label{DTM1}
if $C \cap K = \O, \ C,K \in \kx$ then $\nu(C \sc K) = \nu(C) + \nu(K)$; 
\item \label {DTM2} 
$ \nu(U) = \sup\{ \nu(C) : \ C \se U, \ C \in \kx \} $
 for $U\in\ox$;
\item \label{DTM3} 
$ \nu(F) = \inf\{ \nu(U) : \ F \se U, \ U \in \ox \} $  for  $F \in \cx$.
\end{enumerate}
\end{Definition} 

\noindent
For a closed set $F$, $ \nu(F) = \infty$ iff $ \nu(U) = \infty$ for every open set $U$ containing $F$.

\begin{Definition}\label{TMLC}
A topological measure on $X$ is a set function
$\mu:  \cx \cup \ox  \longrightarrow  [0,\infty]$ satisfying the following conditions:
\begin{enumerate}[label=(TM\arabic*),ref=(TM\arabic*)]
\item \label{TM1} 
if $A,B, A \sc B \in \kx \cup \ox $ then
$
\mu(A\sqcup B)=\mu(A)+\mu(B);
$
\item \label{TM2}  
$
\mu(U)=\sup\{\mu(K):K \in \bcx, \  K \se U\}
$ for $U\in\ox$;
\item \label{TM3}
$
\mu(F)=\inf\{\mu(U):U \in \ox, \ F \se U\}
$ for  $F \in \cx$.
\end{enumerate}
\end{Definition} 

The following two theorems from \cite[Section 4]{Butler:DTMLC} give criteria for a deficient topological measure to be a topological measure or a measure.

\begin{theorem} \label{DTMtoTM}
\begin{enumerate}[label=(\Roman*),ref=(\Roman*)]
\item
Let $X$ be compact, and $\nu$ a deficient topological measure. The following are equivalent:
\begin{enumerate}
\item[(a)]
$\nu$ is a topological measure
\item[(b)]
$\nu(X) = \nu(C)  + \nu(X \sm C), \ \ \ C \in \cx$ 
\item[(c)]
$\nu(X) \le \nu(C)  + \nu(X \sm C), \ \ \ C \in \cx$  
\end{enumerate}
\item
Let $X$ be locally compact, and $\nu$ a deficient topological measure. 
The following are equivalent:
\begin{enumerate}
\item[(a)]
$\nu$ is a topological measure
\item[(b)]
$\nu(U) = \nu(C)  + \nu(U \sm C), \ \ \ C \in \kx, \ \ \ U \in \ox$ 
\item[(c)]
$\nu(U) \le \nu(C)  + \nu(U \sm C), \ \ \ C \in \kx, \ \ \ U \in \ox$
\end{enumerate}
\end{enumerate}
\end{theorem}

\begin{theorem} \label{subaddit}
Let $\mu$ be a deficient topological measure on a locally compact space $X$. 
The following are equivalent: 
\begin{itemize}
\item[(a)]
If $C, K$ are compact subsets of $X$, then $\mu(C \cup K ) \le \mu(C) + \mu(K)$.
\item[(b)]
If $U, V$ are open subsets of $X$,  then $\mu(U \cup V) \le \mu(U) + \mu(V)$.
\item[(c)]
$\mu$ admits a unique extension to an inner regular on open sets, outer regular Borel measure 
$m$ on the Borel $\sigma$-algebra of subsets of $X$. 
$m$ is a Radon measure iff $\mu$ is compact-finite. 
If $\mu$ is finite then $m$ is a regular Borel measure.
\end{itemize}
\end{theorem}

\begin{remark} \label{tausm}
In \cite[Section 3]{Butler:DTMLC} we show that a deficient topological measure $ \nu$ is $\tau$-smooth on compact sets
(i.e. if  a net $K_\alpha \searrow K$ , where $ K_\alpha, K \in \kx$ then  $\mu(K_\alpha) \rightarrow \mu(K)$), 
and also $\tau$-smooth on open sets (i.e.  if a net $U_\alpha \nearrow U$, where $U_\alpha, U \in \ox$ then $\mu(U_\alpha) \rightarrow \mu(U)$). 

A deficient topological measure $ \nu$ is also superadditive, i.e. 
if $ \bigsqcup_{t \in T} A_t \subseteq A, $  where $A_t, A \in \mathscr{O}(X) \cup \mathscr{C}(X)$,  
and at most one of the closed sets (if there are any) is not compact, then 
$\nu(A) \ge \sum_{t \in T } \nu(A_t)$. 

If $ F \in \cx$ and $C \in \kx$ are disjoint, then $ \nu(F) + \nu(C) = \nu ( F \sc C)$.

One may consult  \cite{Butler:DTMLC} for more properties of deficient topological measures on locally compact spaces.
\end{remark}

\begin{Definition}\label{SDTM}
A signed deficient topological measure on a locally compact space $X$ is a set function
$\nu:  \cx \cup \ox  \longrightarrow [ - \infty, \infty ] $  that assumes at most one of $\infty, 
-\infty$ and that is finitely additive on compact sets, inner compact regular on open sets, and outer regular on closed sets, 
i.e. 
\begin{enumerate}[label=(SDTM\arabic*),ref=(SDTM\arabic*)]
\item \label{SDTM1}
If $C \cap K = \O, \ C,K \in \kx$ then $\nu(C \sc K) = \nu(C) + \nu(K);$ 
\item \label{SDTM2} 
$\mu(U)=\lim\{\mu(K):K \in \bcx, \  K \se U\}
$ for $U\in\ox$;
\item \label{SDTM3}
$\mu(F)=\lim\{\mu(U):U \in \ox, \ F \se U\}$ for  $F \in \cx$.
\end{enumerate}
\end{Definition} 

\begin{remark} \label{netcond}

In condition \ref{SDTM2} we mean the limit of the net $\nu(C)$ with the index set $\{ C \in \kx: \ C \se U\}$ ordered
by inclusion. The limit exists and is equal to $\nu(U)$. Condition  \ref{SDTM3} is interpreted in a similar way, 
with the index set  being $\{ U \in \ox: \ U \supseteq C \}$ ordered by reverse inclusion.
\end{remark} 

\begin{remark} \label{byCompacts}
Since we consider set-functions that are not identically $ \infty $ or $ - \infty$, we see that for a signed deficient topological measure $ \nu (\O) = 0$.
If $\nu$ and $\mu$ are (signed) deficient topological measures that agree on $\kx$ (or on $\ox$) then $\nu =\mu$;
if $\nu \le \mu$ on $\kx$ (or on $\ox$) then $\nu  \le \mu$.
\end{remark}

\begin{remark}
For a signed deficient topological measure $\nu$ we may define its total variation, a deficient topological measure $| \nu|$,  by  
\begin{eqnarray*} 
|\nu| (U) = \sup \{ \sum_{i=1}^n |\nu(K_i)| : \  \bsc_{i=1}^n K_i \se U, \ K_i \se \kx,  \, n \in \N \}; 
\end{eqnarray*}
for an open set $U$, 
and for a closed subset $F \se X$ 
\begin{eqnarray*} 
|\nu| (F)  = \inf\{ |\nu| (U) : \ F \se U, \ U \in \ox\}. 
\end{eqnarray*} 
See \cite[Sections 2,3]{Butler:DTMLC} for detail.
\end{remark}

\begin{definition} \label{SDTMnorDe}
We define  $\| \nu \| = \sup \{ | \nu(K)|  :  K \in  \mathscr{K}(X) \} $ for a signed deficient topological measure $\nu$.
\end{definition}
\noindent
If $\mu$ is a deficient topological measure then $ \| \mu \| = \mu(X)$.

\begin{definition} \label{properSDTM}
A signed deficient topological measure $\nu$  is called proper if from $m \le |\nu| $, 
where $m$ is a Radon measure and $| \nu|$ is a total variation of $\nu$, it follows that $m = 0$.

Let $\mu$ be a deficient topological measure, and  $\nu$ be a signed deficient topological measure.
We say that $ \nu$ is absolutely continuous with respect to $ \mu$ 
(and we write $ \nu  \ll \mu$)  if $ \mu(A) = 0$ implies $ |\nu| (A) = 0$.
\end{definition}

\begin{definition} \label{STMLC}
A signed topological measure on a locally compact space $X$ is a set function
$\mu: \mathscr{O}(X) \cup \mathscr{C}(X) \longrightarrow [-\infty, \infty]$  that assumes at most one of $\infty, 
-\infty$ and satisfies the following conditions:
\begin{enumerate}[label=(STM\arabic*),ref=(STM\arabic*)]
\item \label{STM1} 
if $A,B, A \sqcup B \in \mathscr{K}(X) \cup \mathscr{O}(X) $ then
$\mu(A\sqcup B)=\mu(A)+\mu(B);$
\item \label{STM2}  
$\mu(U)=\lim\{\mu(K):K \in \mathscr{K}(X), \  K \subseteq U\}
$ for $U\in\mathscr{O}(X)$;
\item \label{STM3}
$\mu(F)=\lim\{\mu(U):U \in \mathscr{O}(X), \ F \subseteq U\}$ for  $F \in \mathscr{C}(X)$.
\end{enumerate}
\end{definition} 

\begin{remark} \label{DTMsdtm}
Definition \ref{SDTM} was first introduced in \cite{Butler:STMLC}.
Any deficient topological measure (topological measure, signed topological measure) is a signed deficient topological measure.
\end{remark}


\begin{remark} \label{RemBRT}
There is a correspondence between deficient topological measures and certain non-linear functionals, see \cite[Section 8]{Butler:ReprDTM}.
In particular, there is an order-preserving isomorphism between compact-finite 
topological measures on $X$ and quasi-integrals on $C_c(X)$, and $\mu$ is a measure iff the 
corresponding functional is linear (see Theorem 42 in Section 4 of \cite{Butler:QLFLC} and Theorem 3.9 in \cite{Alf:ReprTh} for the first version of 
the representation theorem.)
We outline the correspondence.
\begin{enumerate}[label=(\Roman*),ref=(\Roman*)]
\item \label{prt1}
Given a finite deficient  topological measure $\mu$ on a locally compact space $X$ and $f \in C_b(X)$, define functions on $\r$:
$$ R_1 (t) = R_{1, \mu, f} (t) =  \mu(f^{-1} ((t, \infty) )), $$
$$ R_2 (t) =  R_{2,  \mu, f} (t) =\mu(f^{-1} ([t, \infty) )). $$
Let $r$ be the Lebesque-Stieltjes measure associated with $-R_1$, a regular Borel measure on $ \mathbb{R}$.
We define a functional on $C_0(X)$: 
\begin{align} \label{rfform}
\mathcal{R} (f) & = \int _{\mathbb{R}}  id \,  dr = \int_a^b id \, dr  =   \int_a^b R_1 (t) dt + a \mu(X)  =  \int_a^b R_2 (t) dt + a \mu(X).  
\end{align}
where $[a,b]$ is any interval containing $f(X)$.
If $f$ is nonnegative with $f(X) \subseteq [0,b]$ we have:
\begin{align} \label{rfformp}
 \mathcal{R} (f) = \int_0^b  id \, dr  =   \int_0^b R_1 (t) dt =   \int_0^b R_2 (t) dt.
\end{align}
We call the functional $\mathcal{R}$ a quasi-integral (with 
respect to a deficient topological measure $ \mu$) and write:
\begin{align} \label{RFint}
\int_X f \, d\mu = \mathcal{R}(f) = \mathcal{R}_{\mu} (f) =  \int _{\mathbb{R}}  id \, dr.
\end{align}
\item   \label{RHOsvva}
Functional $\mathcal{R} $ is non-linear. 
By   \cite[Lemma 7.7,  Lemma 7.10, Lemma 3.6, Lemma 7.12]{Butler:ReprDTM}  we have:
\begin{enumerate}
\item
$\mathcal{R} (f) $ is positive-homogeneous, i.e. $\mathcal{R} (cf)  = c \mathcal{R} (f) $ for $c \ge 0$, and $\mathcal{R} (0) =0$. 
\item
if $g\, h = 0 $, where $ g, h \ge 0$ or $g \ge 0, h \le 0$, then $\mathcal{R} (gh) = \mathcal{R} (g) + \mathcal{R} (h)$.
\item 
$\mathcal{R}$ is monotone, i.e. if $ f \le g$  then $\mathcal{R} (f) \le \mathcal{R} (g)$.
\item
$ \mu(X)  \cdot \inf_{x \in X} f(x)  \le \mathcal{R}(f)  \le \mu(X) \cdot \sup _{x \in X} f(x) $ for any $f \in C_0(X)$. 
\end{enumerate}
\item \label{mrDTM} 
A functional $\rho$ with values in $[ -\infty, \infty]$ (assuming at most one of $\infty, - \infty$) and $| \rho(0) | < \infty$ 
is called a d-functional if   
on nonnegative functions it is positive-homogeneous, monotone, and orthogonally additive, i.e. for $f, g \in D(\rho)$ (the domain of $ \rho$) we have: 
(d1) $f \ge 0, \ a > 0  \Longrightarrow  \rho (a f) = a \rho(f)$; 
(d2) $0 \le  g \le f \Longrightarrow  \rho(g) \le \rho(f) $;
(d3) $f \cdot g = 0, f,g \ge 0  \Longrightarrow  \rho(f + g) = \rho(f) + \rho(g)$. 

Let  $\rho$ be a d-functional with $  C_c^+(X) \subseteq D(\rho) \subseteq C_b(X)$. 
In particular, we may take functional $ \rf$ on $  C_0^+(X)$. 
The corresponding
deficient topological measure $ \mu = \mu_{\rho}$ is given as follows: \\
If $U$ is open, $ \mu_{\rho}(U) = \sup\{ \rho(f): \  f \in C_c(X), 0\le f \le 1, supp \  f \subseteq U  \},$\\
if $F$ is closed, $ \mu_{\rho}(F) = \inf \{ \mu_{\rho}(U): \  F \subseteq U,  U \in \mathscr{O}(X) \}$. \\
If $K$ is compact, $ \mu_{\rho}(K) = \inf \{ \rho(g): \   g \in C_c(X), g \ge 1_K \}  
= \inf \{ \rho(g): \   g \in C_c(X), 1_K \le g \le 1 \}. $
(See Definition 33 and Lemma 35 in Section 4 of  \cite{Butler:QLFLC}.)
\end{enumerate}

If given a finite deficient topological measure $\mu$, we obtain $ \mathcal R$, and then $\mu_{ \mathcal R}$, then $ \mu = \mu_{ \mathcal R}$.
\end{remark}

\begin{definition}
Let $X$ be locally compact. A functional $ \rho$ on $\fvix$ is Lipschitz continuous if for every compact $K \se X$ there exists 
a number $N_K$ such that $|  \rho(f) - \rho(g)| \le N_K \| f - g\|$ for all $f, g \in \fcsx$ with support in $ K$.
\end{definition} 
  
The proof of the following theorem can be found, for example, in \cite{Bogachev}, \textsection 7.2.6.
\begin{theorem}  \label{Bogachev}
Let $l$ be a regular $\tau$-smooth measure on a topological space $X$ and let 
$\{ f_\alpha\}$ be  an increasing net of nonnegative  lower semicontinuous functions such that the function
$f = \lim_{\alpha} f_{\alpha}$ is bounded. Then
$$ \lim_{\alpha}  \int_X f_{\alpha}(x) \, l(dx) = \int_X f(x) \, l(dx).$$
\end{theorem}

\begin{remark} \label{LSCfn}
It is easy to see that a left-continuous nondecreasing function (that does not assume $-\infty$) 
or right-continuous nonincreasing function (that does not assume $-\infty$) is lower semicontinuous.
\end{remark}

\begin{corollary} \label{NetInt}
Let $l$ be a regular $\tau$-smooth measure on a topological space $X$. 
Let $\{ g_\alpha\}$ be a decreasing net of nonnegative nonincreasing left-continuous functions converging to $g$.
Suppose one of the functions $\{ g_\alpha\}$ is bounded. Then   
$$ \lim_{\alpha}  \int_a^b g_{\alpha}(x)  \, l(dx) = \int_a^b g(x) \, l(dx).$$
\end{corollary}

\begin{proof}
Suppose $g_{\gamma}$ is bounded above by $M>0$. Then $g_{\alpha} \le M$ for all $ \alpha \ge \gamma$.
Consider functions $f_{\alpha} = M - g_{\alpha}$ and $f= M - g$. By Remark \ref{LSCfn}, functions
$f_{\alpha} $ are lower semi-continuous. 
Apply Theorem \ref{Bogachev} to $f_{\alpha}$.
\end{proof}

\section{Integrals over a set via "restricted" deficient topological measures} \label{IntlSetRestr}

We would like to begin with integration with respect to a deficient topological measure over a set. Integration of a continuous bounded function with respect
to a deficient topological measure
is described in part \ref{prt1} of Remark \ref{RemBRT}, but it will not work for a function restricted to a set, as such a function need not be continuous anymore. 
Another approach to obtain integration over a set would be to integrate the function with respect to a deficient measure that is restricted to a set.  
But if $\mu$ is a deficient topological measure on $X$, one can not obtain a deficient topological measure on $A \in \ox \cup \cx$  by simply 
restricting $\mu$ to $A$, i.e., considering $\mu(A \cap B), \ B \in \ox \cup \cx$. One simple reason for this is that 
the intersection of two arbitrary  sets from $ \ox \cup \cx$ does not in general belong to $ \ox \cup \cx$.   
However, there is still a way to obtain new deficient topological measures by "restriction". 
The next two results are from~\cite[Section 5]{Butler:DTMLC}.

\begin{theorem}  \label{restUdtm}
Let $\mu$ be a deficient topological measure on a locally compact space $X, \ V \in \ox$. 
Define a set function $\mu_V$ on $\ox \cup \cx$ by letting
$$ \mu_V ( U ) = \mu(U \cap V), \ U \in \ox, $$
$$ \mu_V (F) = \inf \{ \mu_V (U): \ F \se U, \ U \in \ox \} , \ F \in \cx.$$
Then $ \mu_V$ is a deficient topological measure on $X$.
\end{theorem}

\begin{theorem} \label{DTMmuF}
Let  $X$ be locally compact, and let $F \se \cx$. 
There exists a deficient topological measure $\mu_F$ on $X$ such that 
\[ \mu_F ( K )  = \mu(F \cap K) \ \ \mbox{for} \ \  K \in \kx, \]
\[ \mu_F(U) = \sup \{ \mu_F(K) : \ K \se U , \ K \in \kx \} \ \  \mbox{for}  \ \ U \in \ox. \]
If $F \in \kx$ then $\mu_F(C) = \mu(F \cap C)$ for every $ C  \in \cx$.
\end{theorem} 



The next theorem states some properties of "restricted"  deficient topological measures $\mu_V$ and $\mu_F$,
given by Theorem \ref{restUdtm} and Theorem \ref{DTMmuF}.

\begin{theorem} \label{RestrDTsv}
Let $X$ be locally compact, $ V \in \ox, F \in \cx$, and $ \mu, \nu$ be deficient topological measures on $X$.
Let $\mu_V$ and $\mu_F$ be deficient topological measures given by Theorem \ref{restUdtm} and Theorem \ref{DTMmuF}. Then
\begin{enumerate}[label=(v\arabic*),ref=(v\arabic*)]  
\item \label{svMUFsup}
$\mu_V(U) = \sup \{ \mu_K(U): \ K \se U, \ K \in \kx \}$ for every $ U \in \ox$.
\item \label{svMUFinf}  
$\mu_F ( C) = \inf \{ \mu_V (C) : \  F \se V, \ V \in \ox \} $  for every $ C \in \kx$.
If $F$ is compact then the equality holds for every $ C \in \cx$.
\item \label{svMUFad}
If $A,B$ are disjoint compact sets, or disjoint open sets, then $ \mu_{A \sc B}  = \mu_A + \mu_B$.
\item \label{svMUAmon}
If $\mu \le \nu$ then $ \mu_A \le \nu_A$  for any set $ A \in \ox \cup \cx$.
\item \label{svMUmoSe}
If $A, B \se \ox \cup \cx, \ A \se B$ then $\mu_A \le \mu_B$.
\item \label{svConeLin}
If $a, b \ge 0$ then $(a \mu + b \nu)_A = a \mu_A + b \nu_A$ for any set $ A \in \ox \cup \cx$.
\item \label{muax}
If $A \in \ox \cup \kx$ then $ \mu_A(X)  = \mu(A) $. If $F \in \cx$ then $\mu_F(X) \le \mu(F)$.
\end{enumerate}
\end{theorem}

\begin{proof}
\begin{enumerate}[label=(v\arabic*),ref=(v\arabic*)]  
\item
Assume first that $\mu_V (U)  = \mu (U \cap V) < \infty$. Given $ \eps >0$, let 
$K \in \kx $ be such that $ K \se U \cap V, \  \mu(U \cap V) - \mu(K) < \eps.$
Note that for any $C \se U$ we have $ \mu_K(C) = \mu(K \cap C)  \le \mu(U \cap V) = \mu_V(U)$, 
which gives $\mu_K(U) \le \mu_V(U)$. 
Since
$ \mu_V(U) - \mu_K(U) \le \mu_V(U)  - \mu_K(K) = \mu(U \cap V) - \mu(K) < \eps$, 
we have: 
 $ \mu_V(U) = \sup \{ \mu_K(U): \ K \se U, \ K \in \kx \}. $
 
Now assume that  $\mu_V (U)  = \mu (U \cap V) = \infty$. 
For each $ n \in \N$ choose $K_n \in \kx$ such that $ K_n \se U \cap V, \ \mu(K_n) \ge n$.
For any compact $C$ such that $ K_n \se C \se U$ we have 
$ \mu_{K_n} (C) = \mu(K_n \cap C) = \mu(K_n) \ge n$, 
and we see that $\mu_{K_n} (U)  \ge n$. Then $ \sup \{ \mu_K(U): \ K \se U, \ K \in \kx \} = \infty = \mu_V(U)$. 
\item
Follows from Theorem \ref{DTMmuF} and the following result from~\cite[Section 5]{Butler:DTMLC}.
\begin{lemma} 
Let $X$ be locally compact,  $\mu$ be a deficient topological measure on $X$, and $F \in \cx$. Then for any $C  \in \kx$  
\begin{align} \label{muFCinf}
\mu(F \cap C) = \inf \{ \mu_V (C) : \  F \se V, \ V \in \ox \}, 
\end{align}
where $\mu_V$ is a deficient topological measure on $X$ from Theorem \ref{restUdtm}.
If $ F \in \kx$ then (\ref{muFCinf}) holds for any $C \in \cx$.
\end{lemma}
\item
By Remark \ref{byCompacts} it is enough to check the equality on compact sets (or open sets). The argument is easy and left to the reader.
\item
Take $ A \in \cx$ (respectively,  $ A \in \ox $). Then  $ \mu_A(K)  \le \nu_A(K) $ for every $ K \in \kx$ 
(respectively, $ \mu_A(U)  \le \nu_A(U) $ for every $ U \in \ox$). 
The statement then follows from Remark \ref{byCompacts}.
\item
Arguing as in part \ref{svMUAmon} we see that  $ \mu_A  \le \mu_B $  if $A, B \in \cx$ or if $A, B  \in \ox$.
If $ F \se U, F \se \cx, \ U \in \ox$, then by part \ref{svMUFinf}  $ \mu_F(K)  \le \mu_U(K) $
for every $ K \in \kx$. By Remark \ref{byCompacts} $ \mu_F \le \mu_U$. 
Consider the last case: 
$ U \se F, \ U \in \ox, F \in \cx$. Let $ V \in \ox$ be such that $ F \se V$.  So $U \se F \se V$. Then 
$\mu_U(K) \le \mu_V(K)$, and by part \ref{svMUFinf}  $\mu_U(K) \le \mu_F(K)$ for every $ K \in \kx$. By Remark \ref{byCompacts} 
$\mu_U \le \mu_F$.
\item
The argument is similar to one for part \ref{svMUAmon}.
\item
Easy to see from Theorem \ref{restUdtm} and Theorem \ref{DTMmuF}.
\end{enumerate}
\end{proof} 

\begin{definition} \label{IntOverA}
Suppose $X$ is locally compact, $\mu$ is a deficient topological measure on $X$, and $A \in \ox \cup \cx$ is such that 
$\mu(A) < \infty$. In particular,  for a finite deficient topological measure $\mu$ on $X$, we take $A \in \ox \cup \cx$. 
Let $\mu_A$ be a deficient topological measure given by Theorem \ref{restUdtm} or Theorem \ref{DTMmuF}.
We define the the quasi-integral (with respect to a deficient topological measure $\mu$) of a function $g \in C_b(X)$ over a set $A$ 
with $ \mu(A) < \infty$  to be
the functional $ \rf_{\mu_A} (g)$ as in formula (\ref{RFint}) (where we use a deficient topological measure $\mu_A$) and we write:
\begin{align}  \label{RNforml}
\int_A g \, d \mu = \int_X g \, d \mu_A =  \rf_{\mu_A} (g). 
\end{align}
\end{definition}

\begin{remark} 
If $A = X$ we have the usual $\rf_{\mu} (g)$. If $\mu$ is a measure, we obtain the standard integral $\int_A g \, d \mu$ over a set $A$ .
To evaluate $ \rf_{\mu_A} (g)$ we may use formulas (\ref{rfform}), (\ref{rfformp}), and part \ref{muax} of Theorem \ref{RestrDTsv}.
For example, if $f \in C_b(X),  f(X) \se [a,b],  E \in \ox \cup \kx$, and $ \mu(E) < \infty$,  we may use: 
\begin{align} \label{IntV}
\int_E f d \mu = a \mu(E) + \int_a^b R_{1, \mu_E, f}(t) \, dt = a \mu(E) + \int_a^b R_{2, \mu_E, f}(t) \, dt,  
\end{align}
where, for example, 
\begin{align*} 
R_{1, \mu_V, f}(t) = \mu_V ( f^{-1}((t, \infty))) = \nu(V \cap  f^{-1}((t, \infty))) \mbox{     for     } E = V \in \ox,
\end{align*} 
\begin{align} \label{IntCR}
R_{2, \mu_C, f}(t) = \mu_C(f^{-1}([t, \infty))) = \mu(C \cap f^{-1}([t, \infty))) \mbox{     for     } E = C \in \kx.
\end{align}
If $ f \in C_0^+(X)$ and $f(X) \se [0,b]$ then for $A \in \ox \cup \cx$ with $ \mu(A) < \infty$ we may use: 
\begin{align} \label{IntVC+}
 \int_A f d \mu = \int_0^b R_{1, \mu_A, f}(t) \, dt= \int_0^b R_{2, \mu_A, f}(t) \, dt.
\end{align}
Let $a' = \sup_E f, \ b' = \inf_E f$. Since $R_{i, \mu_E, f} (t)  = \mu(E)$ for $t \in [a, a')$, and $R_{i, \mu_E, f} (t)  =\O$ for $t \in (b', b]$ (where $i=1,2$),
we may rewrite ( \ref{IntV} ) as:
\begin{align} \label{IntVtoch}
\int_E f d \mu = a' \mu(E) + \int_{a'}^{b'} R_{1, \mu_E, f}(t) \, dt =  a' \mu(E) + \int_{a'}^{b'} R_{2, \mu_E, f}(t) \, dt 
\end{align}
\end{remark} 


\begin{lemma} \label{RaNiReg}
Let $X$ be locally compact, $\mu$ be a deficient topological measure on $X$, and $ g \in C_b(X)$.
We have:
\begin{enumerate}[label=(\roman*),ref=(\roman*)]  
\item
If $ V \in \ox,  \mu(V) < \infty$ then
$$ \int_V g \, d\mu =  \lim  \Big\{ \int_K g \, d\mu : \ K \se U, \ K \in \kx \Big\}. $$
\item 
Suppose $ g \in C(X)$ if $X$ is compact and $ g \in \fvixP$ if $X$ is locally compact. 
If  $F \in \cx, \mu(F) < \infty$ then 
$$ \int_F g \, d\mu =  \lim \Big\{ \int_V g \, d\mu : \ F \se V, \ V \in \ox \Big\}. $$
\end{enumerate}
\end{lemma}

\begin{proof}
Let $g(X) \se [a,b]$. If $ g \in \fvixP$ we take $a =0$.
\begin{enumerate}[label=(\roman*),ref=(\roman*)]  
\item
For any $ K \se V,  K \in \kx$ the function $R_{1, \mu_K, g} (t)$ is right-continuous (see \cite[Lemma 6.3]{Butler:ReprDTM}), 
nonincreasing, so by Remark \ref{LSCfn} 
it is lower-semicontinuous. For each $ t \in [a, b]$
by part \ref{svMUmoSe} of Theorem \ref{RestrDTsv} the net $ \{ R_{1, \mu_K, g} (t): \, K \se V, K \in \kx \} $ (ordered by inclusion) is nondecreasing, and 
by part \ref{svMUFsup} of Theorem \ref{RestrDTsv}
\[ \lim_{K \se V, K \in \kx} R_{1, \mu_K, g} (t) = R_{1, \mu_V, g}(t).\]
Note that $0 \le  R_{1, \mu_V, g}(t) \le \mu(V)$.
Apply  formula (\ref{IntV}) and Theorem \ref{Bogachev} to obtain the statement.
\item
Since $ \mu(F) < \infty$, there exists $ U \in \ox$ such that $ F \se U, \mu(U)< \infty$, and then 
$ \mu(V) < \infty$ for all open sets $ V $ such that $ F \se V \se U$. 
For any $V \in \ox $ containing $F$ the function $R_{2, \mu_V, g}$ is left-continuous on $[a, b]$ (use \cite[Lemma 6.3]{Butler:ReprDTM}), 
nonincreasing, and bounded. The assumptions assure that each set  $g^{-1}([t, \infty)) $ is compact for $ t \in (a, b]$.  
For $ t \in [a,b]$ by part \ref{svMUmoSe} of Theorem \ref{RestrDTsv} the net $ \{ R_{2, \mu_V, g} (t): \,  V \in \ox, F \se V \} $
(ordered by reverse inclusion) is nonincreasing,
and by part \ref{svMUFinf} of Theorem \ref{RestrDTsv}
\[ \lim_{V \in \ox, F \se V} R_{2, \mu_V, g} (t) = R_{2, \mu_F, g}(t).\] 
Applying   formula (\ref{IntV})  and Corollary \ref{NetInt} we obtain the statement.
\end{enumerate}
\end{proof}

\begin{theorem} \label{RaNiDTM}
Let $X$ be locally compact, $\mu$ be a compact-finite deficient topological measure on $X$. 
\begin{enumerate}[label=(\alph{enumi}),ref=(\alph{enumi})] 
\item
Let  $ g \in \fvixP$. Then there exists a compact-finite deficient topological measure $\mu_g$ on $X$ such that 
 $\mu_g(A) = \int_A g \, d\mu $ for every $A \in \ox \cup \cx$ with $ \mu(A) < \infty$, and  
 $\mu_g(K) \le  \| g \| \mu(K)$ for any $ K \in \kx$. Also, 
 if $E \sc D = F$, where $E, D, F \in \cx$ and $ \mu(F) < \infty$, then $ \mu_g(E) + \mu_g(D) = \mu_g(F)$.
\item \label{cfmug}
Let $X$ be compact, $ g \in C(X)$. Then  a set function $ \nu_g$  on $A \in \ox \cup \cx$  given by 
$\nu_g(A) = \int_A g \, d\mu $ is a signed deficient topological measure with finite norm $\| \nu_g \| \le \| g \| \mu(X)$.
If $ g \ge 0$ then 
$\nu_g$ is a deficient topological measure, and $\nu_g = \mu_g$.
\end{enumerate}
\end{theorem}

\begin{proof}
\begin{enumerate}[label=(\alph{enumi}),ref=(\alph{enumi})] 
\item
Consider a set function $ \lambda$ on $ \kx$ given by $ \la(K) = \int_K g \, d\mu $. By part \ref{svMUFad} of Lemma \ref{RestrDTsv}
$\lambda$ is finitely additive on $ \kx$. We take $\mu_g$ to be the positive variation $\lambda^+$ of $\lambda$, so $ \mu_g$ is a deficient topological measure
(see \cite[Section 3]{Butler:DTMLC}). By definition of $ \lambda^+$ and Lemma \ref{RaNiReg} we see that   
$\mu_g(A) = \int_A g \, d\mu $ for every $A \in \ox \cup \cx$ with $ \mu(A) < \infty$. 
For any $K \in \kx$ using part \ref{RHOsvva} of Remark \ref{RemBRT}
and part \ref{muax} of Theorem \ref{RestrDTsv} we have:
\begin{align} \label{mugKest}
 \mu_g(K) = \rf_{\mu_K} (g) \le  \| g \| \mu_K(X) = \| g \| \mu(K),
\end{align}
and we see that $\mu_g$ is compact finite. 

Now suppose $E \sc D = F$, where $E, D, F \in \cx$ and $ \mu(F) < \infty$. 
Let $g(X) \se [0,b]$. For each $ t >0$
$$ g^{-1}([t, \infty)) \cap F =  (g^{-1}([t, \infty)) \cap E)  \sc  ( g^{-1}([t, \infty)) \cap G), $$
and all sets are compact.
Since $ \mu$ is finitely additive on compact sets, we see that 
$$ \int_0^b R_{2, \mu_F, g}(t) \, dt = \int_0^b R_{2, \mu_E, g}(t) \, dt + \int_0^b R_{2, \mu_D, g}(t) \, dt, $$
which (since $\mu(F), \mu(E), \mu(D) < \infty$) means that $ \mu_g(F) =  \mu_g(E) +  \mu_g(D)$.
\item
From part \ref{svMUFad} of Theorem \ref{RestrDTsv} and Lemma \ref{RaNiReg} we see that $\nu_g$ is a signed deficient topological measure.
As in (\ref{mugKest}), $| \nu_g(K)| \le \| g \| \mu(K) \le \| g \| \mu(X)$ for each $ K \in \kx$, so by Definition  \ref{SDTMnorDe}
$ \| \nu_g \| \le \|g \| \mu(X)$. 
If $ g \ge 0$ then $\nu_g = \mu_g$ on $\kx$, so by Remark \ref{byCompacts}  $\nu_g = \mu_g$.
\end{enumerate}
\end{proof}

The next example shows that by Theorem \ref{RaNiDTM} we can obtain a signed deficient topological measure.

\begin{example} \label{RNsigned}
In this example we shall use a "parliamentry" topological measure $\mu$ on  $X = [-4,0]\times [0,4]$:   
given an odd number of distinct points of a compact space $X$, 
we assign a closed or open set which is also solid (i.e. which is connected and  has a connected complement) $\mu$-value 1 if the set contains more than 
half of the points, and  $\mu$-value 0 otherwise. (For more information, see \cite[Remark 6.1]{Aarnes:Pure} or \cite[Introduction, Example b]{AarnesButler}.)
Let $\mu$ be the simple topological measure based on points 
$p=(-4,0), q=(-4,4)$, and $s=(0,4)$. Let $g(x,y) = x$, so $ g(X)  = [-4,0]$. 
Let $K \in \kx$ be the closed triangle with vertices $(-4,0), (-4,4)$, and $(0,0)$,  
so $K$ is a closed solid set with $ \mu(K) = 1$. Note that for any $ t  > -4$ the set $ K \cap g^{-1}([t, \infty))$ is a closed solid set with
$\mu(K \cap g^{-1}([t, \infty))) = 0$, and so
$R_{2, \mu_K, g}(t) = \mu( K \cap g^{-1}([t, \infty))) = 0.$ 
By formula (\ref{IntV}) we see that
$ \nu_g(K) = -4 \mu(K) = -4.$
Thus, $\nu_g$ is a signed deficient topological measure.
Note also that $ \nu_g(K) = -4 = \inf_K g \cdot \mu(K)$.

Now we shall show that $ \nu_g$ is not a signed topological measure.
Let $C \in \kx$ be the closed triangle with vertices $(-4,0), (-2,4)$, and $(0,0)$, and let $U,V$ be the connected components of $X \sm C$. 
Triangles $U, V$ are open solid sets. 
Since $\mu(U) = \mu(V) = 0, \mu(X) = 1$, the sets 
$U \cap  g^{-1}((t, \infty)), V \cap  g^{-1}((t, \infty)) $, and $  g^{-1}((t, \infty))$ are open solid sets with at most one of points $p, q, s$ for each $t >0$, 
by formula (\ref{IntV}) we calculate $ \nu_g(U) = \nu_g(V) = 0, \nu_g(X) = -4$.
For each $t>-4$ the set $C \cap g^{-1}([t, \infty))$ is not solid, but its complement consists of two disjoint open solid sets which we call $U_t$ and $V_t$,
and $\mu(U_t) = 1, \mu(V_t) = 0$. Then $R_{2, \mu_C, g}(t) = \mu( C \cap g^{-1}([t, \infty))) = \mu(X) - \mu(U_t) - \mu(V_t) = 0$ for each $t > -4$. Since
$\mu(C) = 0$, we obtain $\nu_g(C) = 0$. Since  $\nu_g(X) = -4$, while  $ \nu_g(V) = \nu_g(W) = \nu_g(C) = 0$, we see that $\nu_g$ 
is not a signed topological measure.
\end{example} 

\begin{remark} \label{RNsdtm}
If in Theorem \ref{RaNiDTM} $ \mu$ is a measure, then it is easy to see that $\nu_g$ is a signed measure, 
which is a measure if $g \ge 0$. However, the next example shows that 
if $\mu$ is a topological measure, then $\mu_g$ may not be a a topological measure. 
\end{remark}

\begin{example} \label{RNnunotTM}
Let $X = \r^2$, and  $D$ be the disk of radius 2 centered at the origin. Let $g \in \fcsx$ be the function with $ \supp g = D$, whose graph 
is the cone with the vertex $(0,2)$ and the base $D$.
Thus, $g(X) = [0,2]$.
Let $ \mu$ be the simple topological measure based on points $(0,0), (1,0)$ and $(4,0)$ as in \cite[Example 2]{Butler:WaysLC}. 
Since $g^{-1}((t, \infty))$ is a bounded open solid set, by  \cite[Remark 4]{Butler:WaysLC} we see that $R_{1, \mu, g}(t) = 1$ for $0 < t < 1$,
and  $R_{1, \mu, g}(t) = 0$ for $1 \le t \le 2$. Then  $\mu_g(X) = \int_X g \, d \mu = \int_0^4  R_{1, \mu, g}(t)  \, dt = 1$.
Let compact $K = \{ (x,0): 1 \le x \le 4 \} $ and  $V = X \sm K$. 
For each $t \ge 0$ we see that $K \cap g^{-1}([t, \infty))$ is a compact solid set, and  
$V \cap g^{-1}((t, \infty))$ is a
bounded open solid set. Note that $R_{2, \mu_K, g}(t) = \mu(K \cap g^{-1}([t, \infty))) = 0$, and $R_{1, \mu_V, g}(t) = 0$ for each $t >0$.
Thus,   $ \mu_g(K) = \mu_g(V) = 0$. It follows that $\mu_g$ is not a topological measure. 

When $ \mu$ is a topological measure, $\mu_g$ may not be a topological measure even for a strictly positive function $g$. 
Let $X = [1,3]\times [0,1]$.  Let $\mu$ be the "parliamentry" simple topological measure based on points 
$p=(2,0), q=(1,1)$, and $s=(3,1)$. Let $g(x,y) = x$, so $g>0$ and $g(X) = [1,3]$. 
Let $K $ be the triangle with vertices $ (1,1), (1,0),(3,0)$ and let $ V  = X \sm K$. 
So $ K \in \kx, V \in \ox, X =K \sc V$, and $\mu(K) = 1, \mu(V) = 0, \mu(X) = 1.$
Note that for $t \in (1, 3]$ the set $K \cap g^{-1}([t, \infty))$ is a closed solid set containing at most one of the points $p, q, s$, 
and so  $R_{2, \mu_K, g}(t) = \mu(K \cap g^{-1}([t, \infty))) = 0$. 
Similarly,  $R_{1, \mu_V, g}(t) = 0$ for  $t \in (1, 3].$
Using formula (\ref{IntV}) we 
see that $\mu_g(V) =0$ and  $\mu_g(K) =1$. 
By formula (\ref{rfform})  we also have $\mu_g(X) = 2$, since $R_{2, \mu, g}(t)  = 1$  if $t \in [1, 2]$, and 
$R_{2, \mu, g}(t)  = 0$  if $t > 2 $. Thus, $ \mu_g$ 
can not be a topological measure. 
\end{example}

\begin{theorem} \label{rgPr} 
Let $X$ be locally compact, $ \mu$ be a compact-finite deficient topological  measure on $X$.
Let $\mu_g$ (for $ g \in \fvixP$) and $ \nu_g$ (for $g \in C(X)$ and $X$ compact) be as in Theorem \ref{RaNiDTM}.
Then
\begin{enumerate}[label=(b\arabic*),ref=(b\arabic*)]
\item \label{rgA}
If $a \ge 0$ then $ \mu_{ag} = a \mu_g$. If $g=0$ then $\mu_g = 0$. The same holds for $\nu_g$.
\item
If $g\, h = 0 $, where $ g, h \ge 0$  then $ \mu_{g+h} = \mu_g+ \mu_h$. If $  g\, h = 0, \, g \ge 0, h \le 0$ then $ \nu_{g+h} = \nu_g+ \nu_h$. 
\item 
If $ g \le h$ then $ \mu_g \le \mu_h$ (respectively,  $ \nu_g \le \nu_h$).
\item \label{rgInfSup}
If $ A \in \ox \cup \cx $ and $ \mu(A) < \infty$  then for $ \la_g = \mu_g$ or $ \la_g = \nu_g$ 
\begin{align} \label{3sta}
 \mu(A)  \cdot \inf_A  g  \le \la_g(A) \le  \mu(A) \cdot \sup_A g.
\end{align}
In particular, 
\begin{align} \label{lagNorm}
\mu_g (A) \le  \mu(A) \cdot \sup_A g \le \| g \| \mu(A), \ \  | \nu_g(A) | \le  \mu(A)  \sup_A |g| \le  \| g \| \mu(A). 
\end{align} 
\item
If $ A \in \ox \cup \cx , \,  \mu(A) < \infty$  and $g = h $ on $A$ then $\mu_g(A) = \mu_h(A)$ (respectively, $\nu_g(A) = \nu_h(A)$).
\item \label{rgrhohP}
(Lipschitz continuity) If $K \in \kx,  f,g \in \fcsxP $ and $ \supp f , \supp g \se K$ then 
$|\mu_g(K) - \mu_f(K) = | \rf_{\mu_K} (g) -  \rf_{\mu_K} (f) |  \le \| f-g \| \mu(K)$.
\item \label{laglaK}
If $ A \in \ox \cup \kx $ and $g = 1 $ on $A$ then $\mu_g(A) = \mu(A)$ (respectively, $\nu_g(A) = \mu(A)$).
\item \label{laUlagK}
For any open set $U \se X$ we have $ \mu(U)  = \sup \{ \mu_g(K): K \se U, K \in \kx, g=1 \mbox{    on    } K \} $
(respectively, $\mu(U)  = \lim \{ \nu_g(K): K \se U, K \in \kx, g=1 \mbox{    on    } K \} $).
\item \label{lagLinComb} 
If $\mu =c \la+ d \delta$, where $ c,d \ge 0$ and $\lambda, \delta$ are compact-finite deficient topological measures on $X$,  
then $\mu_g = c \la_g +d \delta_g$. 
(Here $\lambda_g, \delta_g$ are deficient topological measures obtained from $\lambda, \delta$ by Theorem \ref{RaNiDTM}.)
Respectively,  $\nu_g = c \la_g +d \delta_g$, where  $\lambda_g, \delta_g$ are signed deficient topological measures obtained 
from $\lambda, \delta$ by Theorem \ref{RaNiDTM}. 
\item
If $ \mu \le \la $  where $\la$ is a compact-finite deficient topological measures on $X$, then $\mu_g \le \lambda_g$ 
(respectively, $\nu_g \le \lambda_g$) 
where $ \la_g$ is a deficient topological measure (respectively, a signed deficient topological measure) obtained
from $\lambda $ by Theorem \ref{RaNiDTM}.  
\item
If $ \mu$ is a proper deficient topological measure, then so is $\mu_g$ (respectivey, $\nu_g$).
If $g >0$ then $ \mu $ is a proper deficient topological measure
iff $\mu_g$ is a proper deficient topological measure.
\item \label{mugMea}
If $\mu $  is a measure then $\nu_g$ is a signed measure, 
which is a measure if $g \ge 0$.  
If $g >0$ then $ \mu $ is a Radon measure iff $\mu_g$ is a Radon measure.
\item 
If $\mu$ is finite and $ g \ge 0$ then $\mu_g(X) = \mu_g(Coz(g))$.
\item \label{absCont}
(Absolute continuity) For any $ \epsilon >0$ there exists $ \delta >0$ such that if $A \in \ox \cup \cx$ and $\mu(A) <\delta$, then
$ \mu_g(A) < \eps$ (respectively, $ | \nu_g(A)| < \eps$). Hence, $ \mu_g \ll \mu$ (respectively, $ \nu_g \ll \mu$). 
\end{enumerate}
If $X$ is compact, $g, h \in C(X)$ we also have:
\begin{enumerate}[label=(\roman*),ref=(\roman*)]
\item \label{withC}
If $c$ is a constant then $\nu_c = c \mu$ and $ \nu_{g + c} = \nu_g + c \mu$.
\item \label{Clevel} 
If $g \le c, \ g = c$ on $ \supp h, \ h \in C(X) $  then $ \nu_{g+h} = \nu_g + \nu_h$.
\item  \label{rgrhoh}
(Lipschitz continuity): $| \nu_g(A) - \nu_h(A) |  = | \rf_{\mu_K} (g) -  \rf_{\mu_K} (h) | \le  \mu (A)  \| g-h \| $ for  $ A \in \ox \cup \cx$.
\end{enumerate}
\end{theorem}

\begin{proof}
\begin{enumerate}[label=(b\arabic*),ref=(b\arabic*)]
\item
Since $ag \in \fvixP$, we may define the deficient topological measure $\mu_{ag}$.
By Remark \ref{byCompacts} it is enough to show that $ a \mu_g = \mu_{ag}$ on $\kx$. 
Let $ K \in \kx$. By formula (\ref{RNforml}) and part \ref{RHOsvva} of Remark \ref{RemBRT} we have:
$ \mu_{ag} (K) = \rf_{\mu_K} (ag) = a \rf_{\mu_K} (g) = a \mu_g(K).$ 
Thus, $ \mu_{ag} = a \mu_g.$ 

If $g=0$ then  by part \ref{RHOsvva} of Remark \ref{RemBRT}$ \mu_g(K)  =\rf_{\mu_K}(0) = 0$ for  each $K \in \kx$. Thus, $\mu_g=0$.
The argument for $ \nu_g$ is the same.
\item 
We need to check that  $ \mu_{g+h} = \mu_g+ \mu_h$ on $\kx$. For $K \in \kx$ by  part \ref{RHOsvva} of Remark \ref{RemBRT} 
$ \mu_{g+h} (K) = \rf_{\mu_K} (g +h) =  \rf_{\mu_K} (g) +  \rf_{\mu_K} (h) = \mu_g(K) + \mu_h(K).$ The proof for $ \nu_g$ is the same.
\item
By Remark \ref{byCompacts} it is enough to check that $ \mu_g \le \mu_h$ (or $ \nu_g \le \nu_h$) on $ \kx$, 
which easily follows from  part \ref{RHOsvva} of Remark \ref{RemBRT}.
\item
Let $A \in \ox \cup \cx, \, \mu(A) < \infty$. By formula (\ref{IntVtoch}) we have 
$ a' \mu(A) \le \int_A f \, d \mu \le   a' \mu(A) + \int_{a'}^{b'} R_{2, \mu_A, f}(t) \, dt  \le   a' \mu(A) + \int_{a'}^{b'} \mu(A) \, dt  = b' \mu(A)$, and 
the statement follows.
\item
If $F \in \cx$ and $g=h$ on $F$ then $R_{2,\mu_F, g}(t) =R_{2,\mu_F, h}(t) $ for every $t >0$.
If $U \in \ox$ and $g=h$ on $U$ then $R_{1,\mu_U, g} =R_{1,\mu_U, h} $ for every $t \ge 0$. 
From formula (\ref{IntVC+}) (applied on $[a,b]$ that contains $f(X) $ and $g(X)$)  we see that $\mu_g(A) = \mu_h(A)$ for $A=F$ or $A=U$.
For $\nu_g$ we apply a similar argument using formula (\ref{IntV}).
\item
Let $K \in \kx$. Using 
\cite[Lemma 7.12(z8)]{Butler:ReprDTM} and part \ref{muax} of Theorem \ref{RestrDTsv} we have:
$ | \mu_g(K) - \mu_f(K)|  = | \rf_{\mu_K}(g) - \rf_{\mu_K}(f) | \le \mu_K(X)  \| g-f\| = \mu(K)  \| g-f\|$.
\item
Follows from formula (\ref{3sta}).
\item
Follows from part \ref{laglaK} and regularity of $\mu$.
\item
It is enough to check that  $\mu_g(K) = c \la_g(K) +d \delta_g(K)$ (respectively, $\nu_g(K) = c \la_g(K) +d \delta_g(K)$)
for each $K \in \kx$, which follows 
easily from formula (\ref{IntVC+}) (respectively, from (\ref{IntV})).
\item
Let $K \in \kx$. By part \ref{svMUAmon} of Theorem \ref{RestrDTsv} $\mu_K \le \la_K$. By  \cite[Theorem 8.7 (II)]{Butler:ReprDTM}
$\mu_g(K) = \rf_{\mu_K} (g) \le  \rf_{\lambda_K} (g) =\lambda_g(K)$. The statement follows.
\item 
Follows from formula (\ref{3sta}) and  \cite[Theorem 4.5(I), Section 4]{Butler:Decomp}.
\item
See first Remark \ref{RNsdtm}.
Now let $g>0$. Assume that $\mu_g$ is a Radon measure. By \cite[Theorem 4.3, Section 4]{Butler:Decomp} 
we may write $\mu = m + \la$ where $m$ is a  Radon measure 
and $\la$ is a proper deficient topological measure. By part \ref{lagLinComb} $\mu_g = m_g + \la_g$. 
But since $\mu_g$ and $m_g$ are both Radon measures, 
by uniqueness of decomposition in  \cite[Theorem 4.3, Section 4]{Butler:Decomp} we see that $ \lambda_g = 0$.  Then from  part \ref{rgInfSup} it follows 
that $\la (K) =0$ for every $K \in \kx$. Then $ \la = 0$ and $\mu =m$ is a Radon measure.
\item
Let $P = Coz g$. Since $R_{1, \mu_P, g}(t) = R_{1, \mu_X, g}(t)$ for all $t \ge 0$,  by formula (\ref{IntVC+})
we  have $\mu_g(P) = \rf_{\mu_P} (g)  = \rf_{\mu_X} (g) = \mu_g(X)$.
\item
Follows from (\ref{lagNorm}).
\end{enumerate}

\begin{enumerate}[label=(\roman*),ref=(\roman*)]
\item
It is enough to show that $\nu_{g+c} = \nu_g + c \mu$ on $\kx$; the other statement then will follow from part \ref{rgA}.
For $ K \in \kx$
by \cite[Lemma 7.12(i)]{Butler:ReprDTM}
and part \ref{muax} of Theorem \ref{RestrDTsv} we have:
\begin{align} \label{vsp}
 \nu_{g+c}(K) = \rf{\mu_K}(g + c) =  \rf{\mu_K}(g) + c \mu_K(X) = \nu_g(K) + c \mu(K).
 \end{align}  
\item
Suppose first $0 \le g \le c, \, g=c$ on $\supp h$.
We need to check that $ \nu_{g+h} = \nu_g + \nu_h$ on $\kx$. Let $K \in \kx$. 
From \cite[Theorem 7.10, Lemma 4.4]{Butler:ReprDTM} it follows that $ \rf_{\mu_K}$ satisfies the $c$-level condition 
(see  \cite[Definition 4.3]{Butler:ReprDTM}, so
$$ \nu_{g+h} (K) = \rf_{\mu_K}(g+h) =  \rf_{\mu_K}(g) +  \rf_{\mu_K}(h) =  \nu_g(K) + \nu_h (K).$$ 
For the general case $g \le c, \, g=c$ on $\supp h$ choose constant $b \ge 0$ such that $0 \le g + b \le c+b$ on $\supp h$.
For any $K \in \kx$ we have: $ \nu_{g+b+h} (K) = \nu_{g+b} (K) + \nu_{h}(K)$, which by (\ref{vsp}) gives
$\nu_{g+h}(K) + b \mu(K) = \nu_g (K) +b \mu(K) + \nu_h(K)$, i.e $\nu_{g+h}(K) = \nu_g (K) + \nu_h(K)$. Thus,   $\nu_{g+h} = \nu_g  + \nu_h$.
\item
 Using 
\cite[Lemma 7.12(ii), ]{Butler:ReprDTM} and  part \ref{muax} of Theorem \ref{RestrDTsv}, as in part \ref{rgrhohP} we obtain the statement.
\end{enumerate}
\end{proof}


\begin{remark}
Example \ref{RNsigned} and  similar examples
show that inequalities in part \ref{rgInfSup} of Theorem \ref{rgPr} can be realized as equalities.
\end{remark}

\begin{theorem} \label{vmSubadd}
Suppose $X$ is locally compact and  $ g \in \fvixP$, or $X$ is compact and $g \in C(X)$.  
Suppose $\mu$ is a compact-finite topological measure on  $X$.
Suppose $A, B , A \sc B \in \ox \cup \cx$, and $ \mu( A \sc B) < \infty$.
Then
\[ \int_{A \sc B} g \, d\mu \le \int_A  g \, d\mu  + \int_B g \, d\mu  +  \mu(A \sc B) \, \max\{ \omega(g, A ),  \omega(g, B) \}, \]
where $\omega(f, A) = \sup_{x  \in A}  f(x) - \inf_{x \in A} f(x)$.
\end{theorem}

\begin{proof} 
Let $g(X) \se [a,b]$ where $a \le 0$.
Suppose first that $D \sc V = U$, where $ D \se \cx, V, U \in \ox$, and $ \mu(U) < \infty$. 
Let $p = \inf_{x \in D} g(x), \, q = \sup_{x  \in D} g(x) $.

If $a< t < p $ then  $ D \se g^{-1}((t, \infty)) $, and so $ g^{-1}((t, \infty)) \cap U  = (g^{-1}((t, \infty)) \cap V ) \sc D$. Thus, 
$R_{1, \mu_U, g}(t) = R_{1, \mu_V, g}(t) + \mu(D)$.  
  
If $ t > q$ then $g^{-1}((t, \infty)) \cap D = \O$, so $ g^{-1}((t, \infty)) \cap U = g^{-1}((t, \infty)) \cap V$, i.e. 
$ R_{1, \mu_U, g}(t) = R_{1, \mu_V, g}(t)$.
Also, 
\begin{align*}
 \int_p^q  (R_{1, \mu_U, g}(t) - R_{1, \mu_V, g}(t)) dt &\le  \int_p^q  R_{1, \mu_U, g}(t) dt \le (q-p)  \mu(U)  \\
 &= \omega(g, D) \mu(D \sc V).
\end{align*}
By superadditivity $ \mu(U) - \mu(V) - \mu(D) \ge 0$. 
Using also (\ref{IntV}) and  (\ref{3sta}) we obtain:  
\begin{align*}
 \int_U g \, d\mu  &-  \int_V g \, d\mu  = a(\mu(U) - \mu(V)) + \int_a^b (R_{1, \mu_U, g}(t) - R_{1, \mu_V, g}(t)) dt  \\
&=   a(\mu(U) - \mu(V))  + \int_a^p \mu(D) dt  +  \int_p^q (R_{1, \mu_U, g}(t) - R_{1, \mu_V, g}(t)) dt   \\   
&\le  a(\mu(U) - \mu(V))  + (p-a) \mu(D) +  \omega(g, D) \mu(D \sc V)\\
&=  a(\mu(U) - \mu(V)- \mu(D))    + p \mu(D)  +  \omega(g, D) \mu(D \sc V)\\
& \le   p \mu(D)  +  \omega(g, D) \mu(D \sc V) \le \int_D g \, d \mu +  \omega(g, D) \mu(D \sc V).
\end{align*}
The statement now follows.

All other possible cases are proved similarly. For example,
if $D \sc V =C$, where $ D, C  \se \cx, \, V \in \ox $ and $ \mu(C) < \infty$, we use the fact that $R_{2, \mu_C, g}(t) = R_{2, \mu_D, g}(t) + \mu(V)$
for $ t < p = \inf_{x \in V} g(x)$, and  $R_{2, \mu_C, g}(t) = R_{2, \mu_D, g}(t) $ for $t >q = \sup_{x \in V} g(x)$.
\end{proof}

\begin{remark}
Theorem \ref{vmSubadd} says that for a deficient topological measure $\mu_g$ or a signed deficient topological measures $\nu_g$ 
obtained in Theorem \ref{RaNiDTM} we have
\begin{align} \label{omegaNer}
 \mu_g(A \sc B) \le \mu_g(A) + \mu_g(B)  +  \mu(A \sc B) \, \max \{ \omega(g, A ),  \omega(g, B) \}, 
\end{align}
\[ \nu_g(A \sc B) \le \nu_g(A) + \nu_g(B)  +  \mu(A \sc B) \, \max \{ \omega(g, A ),  \omega(g, B) \}, \]
whenever  $A, B , A \sc B \in \ox \cup \cx$, and $ \mu( A \sc B) < \infty$.
When $X$ is compact, in \cite[Theorem 5.10]{Svistula:Integrals}  it was shown that 
$\mu_g(A \sc B) \le \mu_g(A) + \mu_g(B)  +  \mu(X) \, \max\{ \omega(g, A ),  \omega(g, B) \}$, 
and it inspired Theorem \ref{vmSubadd}.

By Theorem \ref{DTMtoTM},  $\mu_g$ is a topological measure iff $ \mu_g(U) \le  \mu_g(K)  + \mu_g(U \sm K)$ 
for any open $U$ and any compact $K \se U$. 
Example \ref{RNnunotTM} shows that even when $ \mu$ is a finitely defined simple topological measure and $g$ is strictly positive, 
$\mu_g$ may not be a topological measure. This is a contrast to the situation with measures (see part \ref{mugMea} of Theorem \ref{rgPr}).
Inequality (\ref{omegaNer}) is what we can say in general. Using Theorem \ref{subaddit}, we can also say, that, unless $m_g$ is a measure, 
any inequality that estimates $\mu_g(A \sc B)$ above using $\mu_g(A)$ and $ \mu_g(B)$
(like (inequality \ref{omegaNer})) must have the form 
$\mu_g(A \sc B) \le \mu_g(A) + \mu_g(B)  +  \mu(A \sc B) + \delta(A,B)$ with $\delta(A,B) >0$ for some $A,B \in \ox$ or $A,B \in \kx$.
\end{remark} 
 
\section{Integrals over a set via functionals} \label{IntSetFnl}

In this section we present another interpretation of the deficient topological measure $\mu_g$ given by Theorem \ref{RaNiDTM} in the case where $ \mu$ is finite.

Let $X$ be locally compact. Let $\mu$ be a finite deficient topological measure on $X$ with the corresponding functional
$\rho = \mathcal{R}_{\mu} $ according to part \ref{prt1} of Remark \ref{RemBRT}. Let $g \in C_0^{+}(X)$. 
Let $\rho_g$ be a functional on $C_0(X)$ defined by $\rho_g(f) = \rho(fg)$. Since $\rho$ is a d-functional, so is  $\rho_g$. 

\begin{definition} \label{phiG}
Let $\mugt $ be the deficient topological measure corresponding to $\rg$ by part \ref{mrDTM} of Remark \ref{RemBRT}.  
\end{definition}

\begin{remark} \label{lagFormu}
By part \ref{mrDTM} of Remark \ref{RemBRT}
\begin{align*} 
\mugt (K) &= \inf\{ \rg(f) = \rho(fg): \ f \ge 1_K, f \in \fcsx \} \\
&=  \inf\{ \rg(f) = \rho(fg): \ 1_K \le f \le 1, \, f  \in \fcsx \}, 
\end{align*}
for $ K \in \kx$, and for  $U \in \ox$ 
\begin{align*} 
\mugt (U) = \sup \{ \rg(f) = \rho(fg): 0 \le f \le 1, \  f \in \fcsx, \ \supp f \se U \}.
\end{align*}
\end{remark}

\begin{remark} \label{intKmeas}
If $\mu$ is a regular Borel measure, it is easy to see that 
$$ \mugt (A)  = \int_A g \, d \mu $$ 
for every $ A \in \ox \cup \cx$.
In the next theorem we will show that the equality $ \mugt (A)  = \int_A g \, d \mu  $ also holds when $\mu$ is a finite deficient topological measure.
\end{remark}

\begin{theorem}  \label{intK}
Let $X$ be locally compact. Let $\mu$ be a finite deficient topological measure on $X$.
Let $g \in \fvixP$.  Let $\mugt $ be as in Definition \ref{phiG}, and 
let  $\mu_g $ be the deficient topological measure given by Theorem \ref{RaNiDTM}. 
Then $\mugt = \mu_g$.
\end{theorem}

\begin{proof}
It is enough to show that $ \mugt= \mu_g$ on $\kx$. 
Let $K \in \kx$. Let $ g(x) \se [0,b], \, b \ge 1.$ 
By formula (\ref{IntVC+})
\begin{align} \label{r1}
\mu_g(K) = \rf_{\mu_K}(g) = \int_0^b R_{2, \mu_K, g}(t), 
\end{align}
where $ R_{2, \mu_K, g}(t) $ is given by formula (\ref{IntCR}).
Take any $f \in \fcsx$  such that $ 1_K \le f  \le 1$. 
Then  $[0,b]$ contains the ranges of functions $g$ and $gf$.
By formula (\ref{rfformp})
$$  \rho(fg) =  \int_0^b   R_{2, \mu, fg}(t) \, dt .$$
Note that 
$  g^{-1} ([t, \infty)) \cap K \se (fg)^{-1} ([t, \infty)),$
so 
$ R_{2, \mu_K, g}(t)  \le  R_{2, \mu, fg}(t).$
Then by (\ref{r1}) 
$\mu_g(K) \le \rho(fg) = \rg(f) $.
From Remark \ref{lagFormu}  
we see that   $\mu_g(K) \le  \mugt (K) $.

Now we shall show the opposite inequality. By Lemma \ref{easyLeLC}  choose  $U,V \in \ox$ and  $C = \cl V \in \kx$ such that $ K \se V \se C \se U$.
Let $\K = \{ C \in \kx:  \exists V \mbox{   with   } C = \cl V,  K \se V  \se C \se U \}$. 
$\K$ is a directed set with respect to inverse inclusion and 
$\bigcap_{C \in \K} C = K$. Let $C \in \K$. 
Let $f$ be the Urysohn function such that $1_K \le f \le 1_V$. 
For every $t>0$ we have $(fg)^{-1} ([t, \infty)) \se g^{-1} ([\frac{t}{f(t)}, \infty))  \cap f^{-1} ((0, \infty))  \se  g^{-1} ([t, \infty)) \cap C$,  
so $ R_{2, \mu, fg}(t) \le  R_{2, \mu_C, g}(t)$. Then 
$$ \mugt (K)  \le \rho_g(f) = \rho(fg) =   \int_0^b   R_{2, \mu, fg}(t) \, dt \le  \int_0^b R_{2, \mu_C, g}(t). $$
Each function $R_{2, \mu_C, g}$  is bounded above by $\mu(U)< \infty$. 
Also,  $R_{2, \mu_C, g}$ is left-continuous on $[0, b]$ 
(where left-continuity of $R_{2, \mu_C, g}(t) $ for $ t >0$ follows
from $\tau$-smoothness on compact sets in Remark \ref{tausm}).  
By Corollary \ref{NetInt}
$  \int_0^b R_{2, \mu_C, g} (t) \, dt \rightarrow  \int_0^b R_{2, \mu_K, g}, $ and so 
$ \mugt (K) \le  \int_0^b R_{2, \mu_K, g}(t) \, dt  =  \mu_g(K).$   
\end{proof}


We can say more about the values of the deficient topological measure 
$\mu_g$ using the measure $n_r$ on $X$ given by \cite[Definition 7.15]{Butler:ReprDTM}: 

\begin{theorem} \label{lagINT}
Suppose $\mu$ is a finite deficient topological measure on a locally compact space $X$, $g \in \fvixP$. 
Then
\begin{enumerate}[label=(\roman*),ref=(\roman*)]
\item \label{lagINT1}
For  $E= g^{-1} (A)$,  where $A \se \r$ is open or closed, 
$\mu_g(E) \le \int_E g \, dn_r$.
\item \label{lagINT2}
$ \mu_g(Ker g) = \int_{Ker g}  g \, dn_r =0$, and 
$ \mu_g(E) = \int_{E}  g \, dn_r =0 $
for $E= g^{-1} (A)$, where  $A \se (-\infty, 0)$ is open or closed,
\item \label{lagINT3}
If $\mu$ is a topological measure then 
$\mu_g(E) = \int_E g \, dn_r$
for $E= g^{-1} (A)$, where  $A \se \r  \sm \{0\} $ is closed or $ A \se \r$ is open.
\item \label{lagINT4}
If $X$ is compact and $\mu$ is a topological measure, then 
$ \nu_g(E) = \int_E g \, dn_r$
for $E= g^{-1} (A)$, where  $A \se \r $ is open or closed. 
\item \label{lagINT5}
If $X$ is compact then  $\mu$ is a topological measure on $X$ iff 
$ \mu_g(E) = \int_E g \, dn_r $ 
for any $g \in C^+(X)$ and $E= g^{-1}(C)$ for any closed $C \se \r$.
\end{enumerate}
\end{theorem}

\begin{proof}
\begin{enumerate}[label=(\roman*),ref=(\roman*)]
\item
First, let $D = g^{-1}(C)$, where $C \se \r$ is closed. Let $g(X) \se [0, b]$. Set the measure $p$ to be the restriction
of the measure $n_r$ to $D$. Then 
$ \int_D g \, dn_r = \int_X g \, dp = \int_0^b F(t) dt, $
where $F$ is the distribution function for the measure $p \circ g^{-1}$ defined by 
$$ F(t) = p \circ g^{-1} ([t, \infty)) = p(g^{-1}([t, \infty))) = n_r (g^{-1}([t, \infty)) \cap D). $$
For  $ t >0$ using  \cite[Lemma 7.17]{Butler:ReprDTM} we have:
\begin{align*} 
F(t) &=n_r (g^{-1}([t, \infty)) \cap D) = n_r (g^{-1}([t, \infty)) \cap g^{-1}(C)) \\ 
&= n_r (g^{-1}([t, \infty)\cap C)) \ge \mu(g^{-1}([t, \infty)\cap C))  \\
&= \mu (g^{-1}([t, \infty)) \cap D) = R_{2, \mu_D, g}(t),
\end{align*}
so
\begin{align*}  
\int_D g \, dn_r = \int_0^b F(t) dt \ge \int_0^b R_{2, \mu_D, g}(t) dt = \mu_g(D).
\end{align*}

The case $U = g^{-1}(W)$, where $W \se \r$ is open, is similar. 
We use intervals $(t, \infty)$ instead of $[t, \infty)$  in the distribution function $F$.  Then as above, 
$F(t) \ge R_{1, \mu_U, g} (t)$ and 
$ \int_U g \, dn_r = \int_0^b F(t) dt \ge \int_0^b R_{1, \mu_U, g}(t) dt = \mu_g(U)$.

\item 
Using part \ref{lagINT1}, $ 0 \le \mu_g(Ker g)  \le \int_{Ker g} g \, dn_r = 0$.
If $A \se \r$ is an open or closed subset of $(-\infty, 0)$ then  $E= g^{-1} (A) = \O$. So $\mu_g(E) = \int_{E} g \, dn_r = 0. $
\item
Let  $A \se \r  \sm \{0\} $ be closed or open and  $E= g^{-1} (A)$.
The argument as in the proof of part \ref{lagINT1}, where by  \cite[Lemma 7.17(iii)]{Butler:ReprDTM} inequalities 
become equalities, shows that 
$ \mu_g(E) = \int_E g \, dn_r. $
We are only left to consider  the case $\mu_r(U)$, where $ U = g^{-1}(W)$ for 
an open set $ W \se \r$ such that $0 \in W$. 
Write $W = W_1 \sc \{0\}, \, U_1 = g^{-1}(W_1)$.
Then $U = U_1 \sc Ker g$ and 
\begin{align*}
\mu_g(U) \ge \mu_g(U_1) + \mu_g(Ker g) = \int_{U_1} g \, dn_r+  \int_{Ker g} g \, dn_r = \int_U g \, dn_r.
\end{align*}
Together with part \ref{lagINT1} this gives $\mu_g(U)= \int_U g \,dn_r.$ 
\item
The argument is essentially the same as the one for part \ref{lagINT1}, where by  \cite[Lemma 7.17(iv)]{Butler:ReprDTM}
inequalities become equalities. For example, for $D$ we have 
$F(t) = R_{2, \mu_D, g}(t)$ for any $t \in [a,b] = g(X)$, and then 
$$ \int_D g \, dn_r = a p(X) +  \int_a^b F(t) dt = a \mu(D) +  \int_a^b R_{2, \mu_D, g}(t) dt = \nu_g(D).$$
(Note that when $g \ge 0$, one may also use the argument from \cite[Theorem 20]{Svistula:DTM}.)
\item
The proof is basically one from \cite[Theorem 20]{Svistula:DTM} and is given here for completeness.
Let $Z \se X$ be a zero set. Say, $Z= f^{-1} (0)$. 
We may assume that $0 \le f \le1$.
Let $g = 1 - \displaystyle{\frac1n}f$. Then $g \in C^+(X), \ 1- \displaystyle{\frac1n} \le g \le 1$, and 
$g^{-1} (1) = Z$. Let $U = X \sm Z$. Using formula (\ref{3sta}) and part \ref{lagINT4} we have:
\begin{align*}
(1 - \frac1n) \mu (X) & \le \mu_g (X) = \int_X g \, dn_r =  \int_Z  g \, dn_r +  \int_U  g \, dn_r \\
&= \mu_g (Z) + \mu_g(U) \le \mu (Z) + \mu (U).
\end{align*} 
It follows that $\mu (X) \le \mu (Z)  + \mu (U) = \mu (Z) + \mu (X \sm Z)$.  
Every locally compact space is completely regular, and so the family of all zero sets is a base for closed sets.
From Remark \ref{tausm} 
it follows that  $\mu (X) \le \mu (C)  + \mu (X \sm C)$ for any closed set $C$, and by Theorem \ref{DTMtoTM}
$\mu$ is a topological measure.

\end{enumerate}
\end{proof}

\section{Integration over zero and cozero sets} \label{ZeroInt}

Integration with respect to a deficient topological measure sometimes gives the same results as integration with respect to a measure would give. 
For example, we have the following lemma:

\begin{lemma} \label{Mnullint}
Supppose $X$ is locally compact, $\mu$ is a deficient topological measure on $X, \,  f \in C_b(X)$.
If  $A \in \ox \cup \cx$ is such that $ \mu(A) = 0$ then $ \int_A f  d\mu = 0$.  
\end{lemma}

\begin{proof}
Let $f(X) \se [a,b]$.
Consider first $ V \in \ox$ such that $ \mu(V) = 0$.  Since $R_{1, \mu_V, f}(t) = 0$ for any $t \in [a,b]$, 
using formula  (\ref{IntV}) we see that $ \int_V f  d\mu = 0$. 
For $C \in \cx, \mu(C) = 0$ we may apply a similar argument using  formula (\ref{IntV}) and the fact that $  R_{2, \mu_C, f}(t)  = 0  $.
\end{proof}

Integration with respect to a deficient topological measure sometimes is also very different from
integration with respect to a measure.
If $ \mu$ is a measure and $\mu(Coz(f)) = 0$, then $ \int_X f  d\mu = 0$.  Example \ref{Zfint} below
shows that this is generally false if $ \mu$ is a deficient topological measure.
(On the other hand, part \ref{zf3} of Theorem \ref{Zfintn} below shows that this is true for nonnegative $f$, assuming that 
$\mu(Coz f) < \infty$ (in particular, for finite $ \mu$); this theorem also generalizes Lemma \ref{Mnullint}).  

\begin{example} \label{Zfint}
Let $X=\r, D = [1,4]$ and $ \mu$ be a deficient topological measure such that
 $\mu(A) = 1$ if $ D \se A$ and $\mu(A) = 0$ otherwise, for any $A \in \ox \cup \cx$. (See \cite[Section 6]{Butler:DTMLC}).
Let $f = (|x| -2 )  \wedge 0$, so 
$f(X) = [-2,0]$, $Coz(f) = (-2, 2)$ and $\mu(Coz(f)) = 0$. 
Note that 
$R_{2, \mu, f}(t) = 1$ if $t \in [-2, -1]$ and $R_{2, \mu, f}(t) =  0$ if $ t \in (-1, 0]$. 
By formula (\ref{RFint}) we have: 
$$ \int_X f d \mu = \rf(f) = -2 \mu(X) + \int_{-2}^0 R_{2, \mu, f}(t) dt = -1.$$
Now let $g= -f$, so $g(X) =[0,2]$.
Since $R_{2, \mu, g}(t) = 0$ for every $t>0$, 
we have
$$ \int_X (-f) d \mu = \int_0^2  R_{2, \mu, g}(t) dt = 0.$$
Thus, 
$$   \int_X f d \mu = -1, \ \ \ \ \  \int_X (-f) d \mu = 0.$$
\end{example} 

\begin{theorem} \label{Zfintn}
Suppose $X$ is locally compact, $\mu$ is a deficient topological measure on $X$, $ f \in C_b(X)$. 
Then
\begin{enumerate}[label=(y\arabic*),ref=(y\arabic*)]  
\item \label{zf3} 
If $ f \ge 0$,  then for any open set $V$ with $\mu(V) < \infty$  
$$ \int_V f d \mu = \int_{V \cap Coz(f)} f d \mu,$$ 
and
$$ \int_V f d \mu =0 \Longleftrightarrow \mu(V \cap Coz(f)) =0.$$
In particular, for a finite $\mu$ we have $ \int_X f d \mu = \int_{Coz(f)} f d \mu,$
and
$$ \int_X f d\mu = 0 \Longleftrightarrow   \int_{Coz(f)}  f d\mu = 0 \Longleftrightarrow \mu(Coz(f)) = 0.$$
\item  \label{zf5} 
If  $V \in \ox, \mu(V) < \infty$, then 
$$\mu_V( Z(f)) = \mu(V) \Longrightarrow \int_V f d \mu =0.$$ 
In particular, if $\mu$ is finite and $\mu(Z(f)) = \mu(X)$, then $\int_X f d\mu = 0$.
\item \label{zf6} 
If $X$ is compact, $V \in \ox $, and $ f \le 0$ then 
$$\mu_V( Z(f)) = \mu(V) \Longleftrightarrow \int_V f d \mu =0.$$
In particular,  
$$ \int_X f d\mu = 0 \Longleftrightarrow \mu(Z(f)) =\mu(X).$$
\end{enumerate}
\end{theorem}

\begin{proof}
\begin{enumerate}[label=(y\arabic*),ref=(y\arabic*)]  
\item
Let $ f \in C_b(X), f \ge 0$.
Take any $V \in \ox$ with $ \mu(V) < \infty$, and let $U = V \cap Coz(f)$. 
For $t>0$ we have 
$$U \cap f^{-1}((t, \infty)) = V \cap Coz(f)  \cap f^{-1}((t, \infty)) = V \cap f^{-1}((t, \infty)),$$  
so from (\ref{IntCR}) and ( \ref{IntVC+}) we see that   
$ R_{1, \mu_U, f}(t) = R_{1, \mu_V, f}(t), $ and $ \int_V f \, d\mu = \int_{V \cap Coz(f)} f \, d\mu$.
Now suppose 
 $\int_V f d \mu = \int_0^b R_{1, \mu_V, f}(t) = 0.$
Since the integrand in the second integral is nonnegative, using the right-continuity of $R_{1, \mu_V,f}(t)$ 
(see \cite[Lemma 6.3]{Butler:ReprDTM}),
we must have  $\mu(V \cap Coz (f) ) = R_{1, \mu_V,f}(0) = 0$. 
On the other hand, if $ \mu(V \cap Coz(f)) = 0$, then $R_{1, \mu_V,f}(t) = 0$ for every $t>0$, and by  formula (\ref{IntV})
$\int_V f d \mu = 0$.
\item
Suppose  $f(X) \se [a,b], \, a \le 0, b \ge 0$, and $\mu_V(Z(f)) = \mu(V)< \infty$. 
Using part \ref{muax} of Theorem \ref{RestrDTsv} and superadditivity (see Remark \ref{tausm}), we have
$ \mu(V) = \mu_V(X) \ge \mu_V(Z(f)) + \mu_V(Coz(f)), $ 
and so we must have $ \mu_V(Coz(f)) = \mu(V \cap Coz(f)) = 0$.
With $W =(f^{-1}(t, \infty)) \cap Coz(f)$, we have: $ \mu_V(W) = 0$.
Since  $f^{-1}((t, \infty)) \cap Z(f) = \O$ for $t\ge0$ and  $(f^{-1}(t, \infty)) \cap Z(f) = Z(f)$ for $t < 0$, we see that
 $f^{-1}((t, \infty))  = W$ if $ t \ge 0$, and $f^{-1}((t, \infty)) = Z(f) \sc W $ if $t < 0$.
Then $R_{1, \mu_V, f}(t)  = \mu_V(W) = 0$ for $t \ge 0$ and 
$R_{1, \mu_V, f}(t)  = \mu_V(f^{-1}((t, \infty))) \ge \mu_V(Z(f)) + \mu_V(W) =  \mu_V(Z(f)) = \mu(V) $ for $t <0$. 
Since also  $R_{1, \mu_V, f}(t)  \le \mu(V)$, we have  $R_{1, \mu_V, f}(t)  = \mu(V)$  for $t <0$.
Then
$$ \int_V f d\mu =a\mu(V) + \int_a^b R_{1, \mu_V, f}(t) = a\mu(V) + \int_a^0 \mu(V)  dt =0.$$
\item
Suppose that $X$ is compact.
($\Longrightarrow$)  is given by part \ref{zf5}. 
($\Longleftarrow$) 
Assume that $ \int_Vf d \mu =0$, i.e. by  formula (\ref{IntV})
\begin{align*} 
\int_V f  d\mu =  a \mu (V) +  \int_a^0  R_{2, \mu_V, f}(t) dt = \int_a^0 ( R_{2, \mu_V, f}(t)  - \mu(V)) dt=0.
\end{align*}
The integrand function in the last integral is nonpositive. 
Since $X$ is compact, by \cite[Lemma 6.3(V)]{Butler:ReprDTM}
$R_{2, \mu_V, f}(t)$ is left-continuous at $0$. From left continuity it follows that $R_{2, \mu_V, f}(0) - \mu(V) = 0$. 
But $R_{2, \mu_V, f}(0)  = \mu_V(Z(f))$. 
\end{enumerate}
\end{proof}

\begin{lemma}  \label{zf8}
If $ \mu$ assumes only finitely many values and is not a topological measures then there exists $f \ge 0$ such that 
$\int_X f d\mu = 0, \, \int_X (-f) d\mu <0$.
\end{lemma}

\begin{proof}
Suppose  that $\mu$ is a deficient topological measure, but not a topological measure and that $ \mu$ assumes finitely many values.
By Theorem \ref{DTMtoTM} there are $U \in \ox$ and $C \in \kx, C \se U$ such that $\mu(C)  + \mu(U \sm C) < \mu(U)$.
We may assume that $\cl U$ is compact. (Otherwise, choose by regularity compact $D \se U$ with $ \mu(D) = \mu(U)$, 
then using regularity and Lemma \ref{easyLeLC} choose $U_1 \in \ox$ such that $ \cl{U_1} \in \kx, \,  D \se U_1 \se U, $ and $\mu(U_1) = \mu(D) = \mu(U)$. 
Then $ \mu(C) + \mu(U_1 \sm C) \le \mu(C) + \mu(U \sm C) < \mu(U) = \mu(U_1)$, and we replace $U$ by $U_1$.)
We have
\begin{align} \label{vspom}
 \mu(C) + \mu(U \sm C) + \mu(X \sm U) < \mu(U) + \mu(X \sm U) \le X.
\end{align}
Pick a compact $B \se U \sm C$ with $ \mu(B) = \mu(U \sm C)$. Let $ V = (U \sm C) \sm B$. By superadditivity $ \mu(V) = 0$.
Set a closed set $ G = (X \sm U) \sc B \sc C$. By Remark \ref{tausm} 
$ \mu(G)  = \mu(X \sm U) + \mu(B) + \mu(C) $, which is equal to the left-hand side of (\ref{vspom}), so $ \mu(G) < \mu(X)$.
Now choose $ W \in \ox$ such that $ G \se W, \, \mu(W) = \mu(G)$. Let $\delta = \mu(X) - \mu(W)>0$.

Let $K = X \sm W$. Then $ K \se X \sm G =V \se U$, so $ K$ is compact.
Let $f \in \fcsx$ be a Urysohn function such that $ f=1$ on $K$ and $ \supp f \se V$.
Then $Coz(f) \se V$, so $ \mu(Coz(f) =0$, and by part \ref{zf3} of Theorem \ref{Zfintn}  $ \int_X f \, d\mu = 0$.

Now let $g= -f$, so $ g(X) \se [-1,0]$. For $t \in (-1,0)$ we have $g^{-1}((t, \infty)) \se g^{-1}((-1, \infty)) \se X \sm K =W$, so 
$R_{1, \mu, g}(t) \le \mu(W)$.
Then by formula (\ref{rfform})
$$
 \int_X g \, d\mu = -\mu(X) + \int_{-1}^0 R_{1, \mu, g}(t) dt \le  \int_{-1}^0 (\mu(W) - \mu(X)) dt =-\delta < 0.
$$
\end{proof}

\begin{remark}
Example \ref{Zfint} illustrates Lemma \ref{zf8}.
Also, in Example \ref{Zfint}  $\int_X f d \mu= -1$, while $ \mu(Coz(f)) =0$. This shows that part \ref{zf3}  of Theorem \ref{Zfintn} may not hold 
if the condition $ f \ge 0$ is relaxed.
\end{remark} 

\begin{remark}
Parts \ref{svMUFsup} - \ref{svConeLin} of Theorem \ref{RestrDTsv} for the compact space are stated without proof in \cite[Proposition 5.2]{Svistula:Integrals}.
Some statements from Theorem \ref{RaNiDTM} and Lemma \ref{RaNiReg} are related to \cite[Theorem 5.4]{Svistula:Integrals}. 
Parts \ref{rgA} - \ref{rgInfSup}, \ref{lagLinComb} - \ref{absCont} of Theorem \ref{rgPr} are generalizations of results 
presented in \cite[Theorem 22]{Svistula:DTM} and \cite[Theorem 5.7, Proposition 5.9]{Svistula:Integrals}; for  part \ref{withC} of Theorem \ref{rgPr}
see also  \cite[Theorem 5.7]{Svistula:Integrals}. 
Theorem \ref{intK} is inspired by \cite[Theorem 19]{Svistula:DTM}.
Statements "In particular..."  in Theorem \ref{Zfintn} are generalizations of  \cite[Theorem 3.6, 1-3]{Svistula:Integrals}, 
and Lemma \ref{zf8} generalizes  \cite[Theorem 3.6, 4]{Svistula:Integrals}.  
\end{remark}

\section{Convergence theorems} \label{MCT}

\begin{theorem} \label{MCTdtm}
Let $X$ be a locally compact space, $\mu$ a finite deficient topological measure on $X$. 
\begin{enumerate}[label=(\Roman*),ref=(\Roman*)]  
\item
Suppose $f_{\alpha} \in C_b(X)$ and $f_{\alpha} \nearrow f$ in the topology of uniform convergence. Then 
\[ \int_X f_{\alpha} \, d \mu \longrightarrow \int_X f \, d \mu.\]
\item
Suppose $f_{\alpha} \in \fvixP$ and $f_{\alpha} \searrow f$ in the the topology of uniform convergence. Then 
\[ \int_X f_{\alpha} \, d \mu \longrightarrow \int_X f \, d \mu.\]
\end{enumerate}
\end{theorem}

\begin{proof}
\begin{enumerate}[label=(\Roman*),ref=(\Roman*)]
\item
Since $f \in C_b(X)$, we may consider $\int  f \, d \mu$.
For any $t$, $f_{\alpha}^{-1}((t, \infty)) \nearrow f^{-1}((t, \infty))$, so by Remark  \ref{tausm}
$\mu( f_{\alpha}^{-1}((t, \infty))) \rightarrow  \mu(f^{-1}((t, \infty)))$. i.e. $R_{1,\mu,  f_{\alpha}} (t) \rightarrow R_{1,\mu, f} (t)$.
Applying formula (\ref{rfform}) on $[a,b]$ containing $f(X)$ we see that  
\[ \int_X f_{\alpha} \, d \mu \rightarrow \int_X f \, d \mu.\]
\item
The proof is similar. Note that for any $t >0$ the set $f_{\alpha}^{-1}([t, \infty))$ is compact,  $f_{\alpha}^{-1}([t, \infty)) \searrow f^{-1}([t, \infty))$, and we may 
again apply Remark  \ref{tausm}.  
\end{enumerate}
\end{proof} 

\begin{theorem}
Let $\mu$ be a compact-finite deficient topological measure on  a locally compact space $X$. 
Suppose  ${f_{\alpha}}$ converges uniformly to $f$, $f_{\alpha}, f \in \fcsx, f_{\alpha} \ge 0$.  If $\supp f_{\alpha} ,\supp f \se K$ for some compact $K$, then 
$ \int_X f_{\alpha} \, d \mu  \longrightarrow  \int_X f \, d \mu$. 
If $X$ is compact, $f_{\alpha} \in C(X)$ and  ${f_{\alpha}}$ converges uniformly to $f$ then 
$ \int_X f_{\alpha} \, d \mu  \longrightarrow  \int_X f \, d \mu$.
\end{theorem}
 
\begin{proof}
Follows from parts  \ref{rgrhohP} and \ref{rgrhoh} of Theorem \ref{rgPr}.
\end{proof}

{\bf{Acknowledgments}}:
The author would like to thank the Department of Mathematics at the University of California Santa Barbara for its supportive environment.  

%



\end{document}